\RequirePackage[l2tabu, orthodox]{nag}
\documentclass[10pt]{article}

\usepackage[T1]{fontenc}
\usepackage{lmodern}

\usepackage[a4paper]{geometry}
\usepackage[left,modulo]{lineno}

\usepackage{amsmath}
\usepackage{latexsym}
\usepackage{stmaryrd}
\usepackage{amssymb}
\usepackage{color}
\usepackage{soul}
\usepackage{dsfont,url}
\usepackage{graphicx}
\usepackage{float}

\everymath{\displaystyle}

\usepackage{amsthm}		
\newtheorem{theoreme}{Theorem}[section]

\newtheorem{proposition}[theoreme]{Proposition}

\newtheorem{lemme}[theoreme]{Lemma}
\newtheorem{remarque}[theoreme]{Remark}

\newcommand*{\N}{\mathbb{N}}

\newcommand*{\R}{\mathbb{R}}

\newcommand*{\E}{\mathbb{E}}
\DeclareMathOperator{\e}{e}

\DeclareMathOperator{\drm}{d\!} 
\DeclareMathOperator{\Cov}{Cov}
\DeclareMathOperator{\Var}{Var}

\renewcommand{\P}{\mathbb{P}}

\newcommand{\ind}{\mathds{1}}

\newcommand{\fnct}[5]{#1: \begin{array}{ccc} #2&\longrightarrow &#3\\ #4&\longmapsto &#5\end{array}} 

\newcommand*{\Ncal}{\mathcal{N}}

\newcommand*{\phila}{ \mu }

\newcommand*{\philaa}{ \mu_0 }
\newcommand*{\labee}{ l_0 }

\renewcommand{\leq}{\leqslant}
\renewcommand{\geq}{\geqslant}
\renewcommand{\epsilon}{\varepsilon}
\newcommand*{\abs}[1]{\left \vert#1 \right \vert} 

\definecolor{violet}{rgb}{0.6, 0.4, 0.8}

\title{Risk assessment  using suprema data\thanks{This work was supported by the LABEX MILYON (ANR-10-LABX-0070) of Université de Lyon, within the program "Investissements d'Avenir"(ANR-11-IDEX-0007)
opera operated by the French National Research Agency (ANR).}}
\author{Christophette Blanchet-Scalliet\thanks{University of Lyon, CNRS UMR 5208, Ecole Centrale de Lyon, Institut Camille Jordan, France, christophette.blanchet@ec-lyon.fr}
        \and
        Diana Dorobantu\thanks{University of Lyon, University Lyon 1, ISFA, LSAF (EA 2429) France, diana.dorobantu@univ-lyon1.fr}
        \and
        Laura Gay\thanks{University of Lyon, CNRS UMR 5208, Ecole Centrale de Lyon, Institut Camille Jordan, France, laura.gay@ec-lyon.fr}
        \and 
        Véronique Maume-Deschamps\thanks{University of Lyon, CNRS UMR 5208, Institut Camille Jordan, France, veronique.maume@univ-lyon1.fr}
        \and
        Pierre Ribereau \thanks{University of Lyon, CNRS UMR 5208, Institut Camille Jordan, France, pierre.ribereau@univ-lyon1.fr}
        }
\date{}

\begin{document}
\maketitle

\begin{abstract}
\noindent This paper proposes a stochastic approach to model temperature dynamic and study related risk measures. The dynamic of temperatures can be modelled by a mean-reverting process such as an Ornstein-Uhlenbeck one. In this study, we estimate the parameters of this process thanks to daily observed suprema of temperatures, which are the only data gathered by some weather stations. The expression
of the cumulative distribution function of the supremum is obtained thanks to the law of the hitting time. The parameters are estimated by a least square method quantiles based on this function. Theoretical results, including mixing property and consistency of model parameters estimation, are provided. The parameters estimation is assessed on simulated data and performed on real ones. Numerical illustrations are given for both data. This estimation will allow us to estimate risk measures, such as the probability of heat wave and the mean duration of an heat wave.
\end{abstract}

 \vspace{0.3cm}

\section{Introduction}
\label{section1}
Forecasting and assessing the risk of heat waves is a crucial public policy stake. It requires measure tools in order to evaluate the probability of heat waves and their severity. For example, the paper \cite{siliverstovs2010climate} is interested in assessing the likelihood of occurrence of the heat wave of 2003. For that purpose, they model annual maximum temperatures thanks to mean monthly data.  However, the 
available information depends on meteorological stations. Daily extremes (maximum and / or minimum) might be the only available data. Since temperature does not deviate from its mean level, a mean-reverting process such as an Ornstein-Uhlenbeck (OU) process is commonly  used to model temperature process (see \cite{Dischel1998a}, \cite{Dischel1998b} for example). The authors of \cite{DornierQueruel2000} and \cite{AlatonDjehicheStillberger2001} propose to use an ARMA version of the OU process while \cite{BrodySyrokaZervos2002} propose a fractional Brownian motion (to take into account the long range dependence) instead of the classical Brownian motion in the OU process. In \cite{chaumont2006equilibrium}, the OU process is the basic model used to model the local temperature (of air, of ocean water).  \\
The main purpose of this paper is to estimate the parameters of this OU process. Estimation of OU parameters has been done using observations of the process (see \cite{franco2003maximum}) or more recently using hitting time data in \cite{mullowney2008parameter} for the neuronal activity. However, weather stations do not record either of these data. That's why we propose an estimation based on daily observed suprema of temperatures. Once the parameter estimation is done, 
risk measures related to heat waves may be obtained from Monte Carlo simulations of the dynamic of temperatures with the estimated parameters. For example, we would like to estimate the 
probability of heat waves, namely the probability for outdoor air temperature to exceed a threshold (26.67$^\circ$C during 3 days, see \cite{gringorten1968estimating}) or two thresholds (one 
during night and one during day, see \cite{meteofrance}). Other interesting measures would be the corresponding expected area over the threshold or the mean time over the threshold.\\
Recently, lots of results on the first passage time 
of the process have been obtained. In \cite{repalilipatie}, different expressions for the density function of the first hitting time to a fixed level by an OU process are given.  Since hitting time and suprema are related, the cumulative distribution function (cdf) of the supremum is obtained.\\
 Unlike classical quantile estimation (such as done in \cite{castillo2005} or \cite{Peka}), we do not use the cdf inverse and propose though a new approach to estimate the parameters. Thanks 
to the cdf, we perform a least square method to estimate the OU parameters.\\
The paper is organized as follows. In the next section, the estimation problem is presented. Section 3 is concerned with the theoretical tools. Finally, Section 4 is devoted to the numerical illustrations of the estimation and the related risk measures thanks to the only available data : the daily suprema of temperatures.

\section{Estimation Problem}
\label{sec:estimprob}
Since temperature does not deviate from its mean level, a mean-reverting process such as an OU process is commonly used to model temperature process (see \cite{Dischel1998a}, \cite{Dischel1998b} for example). Here, we use a stationary OU process. 
The temperature variations process $X=(X_t)_{t \geq 0}$ is given by :
\[ \drm X_t = l_0\beta_0(\mu_0 - X_t)\drm t + \sqrt{\beta_0}\drm B_t ,~~~X_0 \sim \Ncal\left(\mu_0,\frac{1}{2l_0}\right)\] 
where $\mu_0 \in \R, l_0, \beta_0 \in \R_+^*$, and $(B_t)_{t \geq 0}$ is a standard Brownian motion. Suppose that $X_0$ and $(B_t)_{t \geq 0}$ are independent. We recall that the measure 
$\Ncal\left(\mu_0,\frac{1}{2l_0}\right)$ is the stationary measure. This modelling is reasonable, as, in the applications, we consider observations only from a sub-period of annual observations (e.g. from summer). Let us note $\theta_0=( \beta_0,\mu_0, l_0)$. We say that $X$ is a stationary OU with parameter $\theta_0$.

The parameter $\mu_0$ is the mean of the stationary process. The parameter $\sqrt{\beta_0}$ is the volatility of the process. It indicates the degree of variation. For the temperature, it reveals a tendency to change quickly and unpredictably. If $\beta_0=0$, the process is purely deterministic and well-known then. Finally, the parameter $l_0$ shows the "speed" of mean-reversion. The parameter $l_0\beta_0$ is sometimes called the relaxation parameter. If $l_0=0$, the process is just a Brownian motion, standard if $\beta_0=1$. The influence of these parameters is shown on Figure \ref{influparam} in Appendix \ref{annexe}.

Let us note for $s,r \in \R^+$, $S_{[s,r[}=\sup_{s \leq t < r} X_t$ and $I_{[s,r[}=\inf_{s \leq t < r} X_t$.

Assume that we observe the suprema on a period $[0,T]$ with a partition $(t_i)_{i\geq 0}$ of constant step $h \geq 0$. We then have $n$ suprema $S_{[t_{i-1},t_i[}$ for $i \in \llbracket 1, n \rrbracket$ on disjoint intervals.  Let us remark here that in our problem of daily observations we will take $h=1$. 

Classical estimation methods are not well suited for the parameter estimation from the supremum observations. Indeed, the likelihood maximization requires the probability density function of the supremum and in order to use quantile methods, one needs to know the supremum's cdf inverse. These two functions can only be obtained  by numerical approximations that are more time consuming than numerical methods to get the cdf itself.

This is why we propose to use the cdf of the supremum, denoted $F^*$, whose expression is given in Proposition \ref{prop_cdf_full}.

Let $N_q \in \N^*$ and $s_j$, $j=1\/,\ldots\/, N_q$ be real numbers. Let us denote $F^*_n$ the empirical distribution function on the sample  $S_{[t_{i-1},t_i[}$, $i=1\/,\ldots\/,n$. We recall that, for $t\in\R, F_n^*(t)=\frac{1}{n} \sum_{i=1}^n \ind_{S_{[t_{i-1},t_i[} \leq t}$. 

A way to estimate $\theta_0$ is to use a least square method by minimizing the sum of squares of the differences between theoretical and empirical cdf. Then, we want to minimize the following function $Q_n$ : 
\[ Q_n\left(\theta\right)=\sum_{j=1}^{N_q}\left[F^*\left(s_j,\theta,h\right)-F_n^*(s_j)\right]^2 \]
where $\theta$ is the parameter of the OU process, $F^*(a,\theta,h)=\P\left(S_{[0,h[}\leq a\right)$ and $s_1\/,\ldots\/, s_{N_q}$ are real numbers (to be chosen later).

Thus, $\theta_0$ is estimated by 
\begin{equation}
\label{minimisation_problem}
\widehat{\theta_n}=\left(\widehat{\beta_n},\widehat{\mu_n},\widehat{l_n}\right)= \underset{\theta \in \R \times \R_+^* \times \R_+^*}{\text{argmin}} ~Q_n\left(\theta\right)
\end{equation}

\begin{remarque}
The problem is stated here with suprema but the same reasoning  may be applied to the infima (or both infima and suprema) to deduce the estimation.
\end{remarque} 

\section{Theoretical Tools}

In this section, we present some useful results to estimate the parameters.

\subsection{Cdf of the supremum}

To minimize the function $Q_n$, we need to compute the cdf $F^*$ of the supremum. As the cdf of the supremum is directly linked with the one of the hitting time, we can find $F^*$ thanks to \cite{repalilipatie}.

\begin{proposition}\label{prop_cdf_full}
For $t\in \R^+$ and $a\in \R$, the cdf $F^*$ of the supremum of the stationary OU process $X$ with parameter $\theta=(\beta,\mu,l)\in \R\times\R_+^*\times \R_+^*$ is given by 
\begin{align*}
F^*(a,\theta,t)&=\P\left(S_{[0,t[} \leq a\right)\\
&= \Phi\left(\left(a-\phila\right)\sqrt{{2l}}\right) -\int_{-\infty}^a \int_{0}^{t\beta}\e^{-\frac{l}{2}\left[ \left(a-\mu\right)^2+\left(x-\mu\right)^2-u\right]}\frac{a-x}{\pi\sqrt{\frac{2u^3}{l}}}\e^{-\frac{(a-x)^2}{2u}}\E\left[\e^{-\frac{l^2}{2}\int_0^u(r_s-a+\mu)^2\drm s}\right] \drm u \drm x
\end{align*}
where $r$ is a 3-dimensional Bessel bridge over the interval $[0,u]$ between 0 and $a-x$ and $\Phi$ is the cdf of the standard normal distribution.
\end{proposition}

We first need the following lemma.

\begin{lemme}
For $x\in \R$, $t\in \R^+$ and $a>x$, the cdf $F^c$ of the conditional supremum of the OU process $X$ with parameter $\theta=(\beta,\mu,l)\in \R\times\R_+^*\times \R_+^*$ starting at $X_0=x$ is given by
\begin{align*}
F^c(a,\theta,t,x)&=\P\left(S_{[0,t[}\leq a ~|~X_0=x\right)\\
&=1-\int_0^{t\beta}\e^{-\frac{l}{2}\left[ \left(a-\mu\right)^2-\left(x-\mu\right)^2-u\right]}\frac{a-x}{\sqrt{2\pi u^3}}\e^{-\frac{(a-x)^2}{2u}}\E\left[\e^{-\frac{l^2}{2}\int_0^u(r_s-a+\mu)^2\drm s}\right] \drm u
\end{align*}
where $r$ is a 3-dimensional Bessel bridge over the interval $[0,u]$ between 0 and $a-x$.\\
For $a\leq x$, $F^c(a,\theta,t,x)=0$.
\end{lemme}

Using the hitting time density for an OU process (see \cite{repalilipatie}), we can deduce this result on the conditional cdf.

\begin{proof}
Let $a >x$ be given and fixed.

Let set $U_t=X_{\frac{t}{\beta}}-\mu$ and $W_t=\sqrt{\beta}B_{\frac{t}{\beta}}$ (which is thus a standard Brownian motion). Then the dynamic of $(U_t)_{t \geq 0}$ is
\[\drm U_t = -l U_t\drm t + \drm W_t ,~~~U_0 =u_0=x-\mu \in \R.\]
For $b >u_0$, we introduce the first passage time $H_b=\inf\{s \geq 0 ; U_s =b \}$.

Since $\P\left(\sup_{0 \leq u < t}U_u \leq b  ~|~U_0=u_0\right)=\P(H_b>t ~|~U_0=u_0)$, we have : 
\[\P\left(S_{[0,t[} \leq a ~|~X_0=x\right) = \P\left(\sup_{0 \leq u < t\beta}U_u \leq a-\mu ~|~U_0=u_0\right)=\P\left(H_{a-\mu}>t\beta ~|~U_0=u_0\right).\]
We conclude using the density  $f_{H_b}$ of $H_b$ given (in \cite{repalilipatie}) by  :
\[f_{H_b}(u)=\e^{-\frac{\mu}{2}(b^2-u_0^2-u)}\frac{b-u_0}{\sqrt{2\pi u^3}}\e^{-\frac{(b-u_0)^2}{2u}}\E\left[\e^{\frac{-\mu^2}{2}\int_0^u(r_s-b)^2\drm s} \right]\]
where $r$ is a 3-dimensional Bessel bridge over the interval $[0,u]$ between 0 and $b-u_0$.
Then, we have
\begin{align*}
\P\left(S_{[0,t[} \leq a ~|~X_0=x\right) &=1-\int_0^{t\beta}\e^{-\frac{l}{2}\left[\left(a-\mu\right)^2-\left(x-\mu\right)^2-u\right]}\frac{a-x}{\sqrt{2\pi u^3}}\e^{-\frac{(a-x)^2}{2u}}\E\left[\e^{\frac{l^2}{2}\int_0^u\left(r_s-a+\mu\right)^2\drm s} \right] \drm u
\end{align*}
with $r$ a 3-dimensional Bessel bridge over the interval $[0,u]$ between 0 and $a-x$.
\end{proof}

\begin{remarque} Similarly, we can obtain the cdf $F_c$ of the conditional infimum of $X$. For $x\in \R$, $t\in \R^+$, $a<x$ and $\theta \in \R\times\R_+^*\times \R_+^*$, we have
\begin{align*}
F_c(a,\theta,t,x)&=\P\left(I_{[0,t[} \leq a ~|~X_0=x\right)\\
&=\int_0^{t\beta}\e^{-\frac{l}{2}\left[ \left(a-\mu\right)^2-\left(x-\mu\right)^2-u\right]}\frac{x-a}{\sqrt{2\pi u^3}}\e^{-\frac{(a-x)^2}{2u}}\E\left[\e^{-\frac{l^2}{2}\int_0^u(r_s+a-\mu)^2\drm s}\right] \drm u
\end{align*}
 with $r$ a 3-dimensional Bessel bridge over the interval $[0,u]$ between 0 and $x-a$. 
 
 For $x\leq a$, $F_c(a,\theta,t,x)=0$.
\end{remarque}

\begin{proof}[Proof of Proposition \ref{prop_cdf_full}]
Integrating with respect to the law of $X_0$, we can express the cdf $F^*$ of the supremum of the OU process $X$ with parameter $\theta=(\beta,\mu,l)\in \R\times\R_+^*\times \R_+^*$.
\begin{align*}
F^*(a,\theta,t)&=\P\left(S_{[0,t[} \leq a\right)\\
&=\int_{-\infty}^a F^c(a,\theta,t,x) \drm x \\
&= \Phi\left(\left(a-\phila\right)\sqrt{{2l}}\right) -\int_{-\infty}^a \int_{0}^{t\beta}\e^{-\frac{l}{2}\left[ \left(a-\mu\right)^2+\left(x-\mu\right)^2-u\right]}\frac{a-x}{\pi\sqrt{\frac{2u^3}{l}}}\e^{-\frac{(a-x)^2}{2u}}\E\left[\e^{-\frac{l^2}{2}\int_0^u(r_s-a+\mu)^2\drm s}\right] \drm u \drm x
\end{align*}
\end{proof}

\begin{remarque}
Note that the cdf $F^*$ is decreasing with respect to $\beta$. Numerically, the cdf seems to be decreasing with respect to $l$ and $\mu$.
\end{remarque}

\subsection{Mixing property}
In order to get statistical properties of estimators, some mixing properties are usually required. Indeed, statistics beyond independence have received a deep attention from the 90's. Mixing is used instead of independence and results such as Laws of Large Numbers or Central Limit Theorems may still hold. There is a very large literature on that subject and we refer to \cite{billingsley,doukhan,rio} and the references therein for definitions and main results. 

Roughly speaking, mixing properties of a process $(Y_t)_{t\in\mathbb{R}}$ quantify the convergence to $0$ as $r$ goes to infinity of
\[\mbox{Cov}\left(f(Y_{z_1}\/,\ldots\/, Y_{z_j})\/,g(Y_{z_{j+1}}\/,\ldots \/, Y_{z_\ell})\right) \]
for $f$ and $g$ in an appropriate class of measurable functions and $0<z_1 < \cdots < z_j \leq z_j+r\leq z_{j+1} <\cdots < z_\ell$. The following proposition means that $(S_{[s\/,t[})_{0\leq s<t}$ is exponentially $\rho$-mixing.

\begin{proposition}[Mixing property]
Let us consider an OU with parameter $\theta=(\beta,\mu,l)$. For any $s,r \geq 0$, for any function $f:\mathcal{C}^0\left([0,s],\R\right) \rightarrow \R$, 
$g:\mathcal{C}^0\left([s+r,+\infty],\R\right) \rightarrow \R$ such that $f$, $g$ are square-integrable with respect to the law of $S$,  and for any $0\leq u\leq s \leq s+r\leq v$,  we have
\begin{equation*}
\abs{\Cov\left[f\left(\left(S_{[0,u[}\right)_{u \leq s}\right),g\left(\left(S_{[s+r,v[}\right)_{v \geq s+r}\right)\right]} \leq \e^{-l\beta r} \sqrt{\Var\left[f\left(\left(S_{[0,u[}\right)_{u \leq s}\right)\right]\Var\left[g\left(\left(S_{[s+r,v[}\right)_{v \geq s+r}\right)\right]}
\end{equation*}
\end{proposition}
\begin{proof}
We can easily adapt the proof of Theorem 2.1 in \cite{gobet2016parameter} in  the one-dimensional case of the OU process satisfying the equation \[ \drm X_t = l\beta(\mu - X_t)\drm t + \sqrt{\beta}\drm B_t ,~~~X_0 \sim \Ncal\left(\mu,\frac{1}{2l}\right)\]

Then, keeping the notation from \cite{gobet2016parameter}, one may take $\varphi = f \circ \sup $ and $\phi= g \circ \sup$ which are square-integrable with respect to the law of $X$ by hypothesis.\end{proof}

\subsection{Consistency of the estimation}
Following the idea of the proof of Theorem II.5.1 in \cite{antoniadis1992regression}, we may prove the consistency of our estimation of the parameter $\theta_0$ provided that the $s_j$, $j=1\:,\ldots\/,N_q$ are chosen such that the function
\begin{equation*}
\fnct{\Psi}{\R \times \R_+^* \times \R_+^*}{[0,1]^{N_q}}{\theta }{\left(F^*(s_j\/,\theta,h)\right)_{j=1\/,\ldots\/,N_q}}
\end{equation*}
is injective and that the parameter $\theta$ belong to a compact subset of $\R \times \R_+^* \times \R_+^*$. 

\begin{proposition}\label{prop:consistency}
Consider an OU process with parameters $\theta_0= (\beta_0,\mu_0,l_0)$. Assume that the parameters $\theta_0$ belong to $\Theta$ a compact subset of $\R \times \R_+^* \times \R_+^*$. 
For any $n\in\mathbb{N}^*$, let $\widehat{\theta}_n=\left(\widehat{\beta}_n,\widehat{\mu}_n,\widehat{l}_n\right)$ be given by (\ref{minimisation_problem}). Then, any a.s. limit point $\theta^*$ of $(\widehat{\theta}_n)_{n\in\mathbb{N}^*}$ satisfies $\Psi(\theta^*)=\Psi(\theta_0)$.
\end{proposition}
\begin{proof}
We adapt the lines of the proof of Theorem II.5.1 in \cite{antoniadis1992regression} and use the ergodic theorem for mixing sequences (see \cite{billingsley} e.g.). 

We denote
\[\widehat{\theta}_n = \text{argmin}_{\theta\in\Theta} Q_n(\theta) \/.\]
Since $\Theta$ is a compact set, the sequence $\widehat{\theta}_n$ has  limit points. Let $\theta^*$ be any limit point of $\widehat{\theta}_n$. Since no confusion can be made, and in order to 
simplify notations, we use $F^*(\cdot,\cdot)$ instead of $F^*(\cdot,\cdot,h)$ for this proof. For $j=1\/,\ldots N_q$, let $\epsilon_n(s_j) = F^*(s_j\/,\theta_0)-F_n(s_j)$ we write
\begin{equation}\label{eq:consistency}
Q_n(\theta) = \sum_{j=1}^{N_q} \epsilon_n(s_j)^2 + \sum_{j=1}^{N_q} (F^*(s_j\/,\theta_0) - F^*(s_j\/,\theta))^2 - 2 \sum_{j=1}^{N_q} \epsilon_n(s_j)(F^*(s_j\/,\theta_0) - F^*(s_j\/,\theta)) \/.
\end{equation}
The ergodic theorem for mixing sequences implies that $\epsilon_n(s_j)$ goes to $0$ a.e. as $n$ goes to infinity for $j=1\/,\ldots\/,N_q$. Now,
\[Q_n(\widehat{\theta}_n) \leq Q_n(\theta_0)\/.\]
Let $n_k$ be a subsequence such that $\widehat{\theta}_{n_k}$ goes to $\theta^*$, using (\ref{eq:consistency}), we have
\[Q_{n_k}(\widehat{\theta}_{n_k})  \longrightarrow \sum_{j=1}^{N_q}  (F^*(s_j\/,\theta_0) - F^*(s_j\/,\theta^*))^2 \ \mbox{a.e.}\]
and $Q_n(\theta_0) \longrightarrow 0$ a.e. We deduce that 
\[F^*(s_j\/,\theta_0) - F^*(s_j\/,\theta^*) =0 \/, \ j=1\/,\ldots\/, N_q\]
which gives the announced result.
\end{proof}
\begin{remarque}
Of course, if the application $\Psi$ is injective, then Proposition \ref{prop:consistency} implies that $\theta^*=\theta_0$ and thus $\widehat{\theta}_n$ goes to $\theta_0$ a.s. as $n$ goes 
to infinity. From some numerical tests, it seems that the injectivity is satisfied.
\end{remarque}
\section{Numerical Applications}
In this section, we want to estimate the parameters, first on some simulated data, and then on real ones. First of all, we need to implement the cdf of the supremum of the OU process $X$ with 3 parameters.

We want to make an estimation on real data with daily suprema observations. This is why without any precision, $h$ will be equal to $1$ for numerical applications in the rest of the paper.
\subsection{Cdf Numerical Computation}

We describe here the used method for the numerical computation. Contrary to what is written in \cite{repalilipatie}, the process $r$ is the unique solution of the following SDE (\cite{Downes2008})
\[\drm r_s=\left(-\frac{r_s}{u-s}+\frac{a-x}{(u-s)\tanh\left(\frac{r_s(a-x)}{u-s} \right)}\right)\drm s+\drm B_s,~~~~ 0<s<u,~~~~ r_0=0.\]
Since the process starts from 0 here, the Euler scheme cannot be applied for this SDE. Recall that the process  $(r_s)_{s \leq u}$ with $r_0=0$ and $r_u=a-x$ and the process $(\tilde{r}_s)_{s \leq u}$ defined by $\tilde{r}_s=r_{u-s}$ with $\tilde{r}_0=a-x$ and $\tilde{r}_u=0$ have same distributions (Exercise XI.3.7 of \cite{revuzyor}).
Therehence, we can use the Euler scheme on the switched Bessel bridge $(\tilde{r}_s)_{s \leq u}$ which verifies the SDE (\cite{revuzyor})
\[\drm \tilde{r}_s=\left(-\frac{\tilde{r}_s}{u-s}+\frac{1}{\tilde{r}_s}\right)\drm s+\drm B_s,~~~~ 0<s<u,~~~~ \tilde{r}_0=a-x.\]
Finally, the integrals are computed by considering the corresponding Riemann sum and the expectation by a Monte-Carlo method with $M=10000$ simulations.\\
The code is written in C++ and the evaluation of the function is very long. Consequently, we had to make a parallel code. Yet, the function "rand" in C++ is not thread safe. Thus, we propose to use the Mersenne Twister generator for the simulation of the random numbers. With the parallelisation, the time for one evaluation of the function $Q_n$ to be minimized has been divided approximately by 5 but is still long (around 52 secs for $N_q=4$, $\theta=(47.5,22,0.02)$, with a 40 cores machine). The duration is not a problem since the optimisation needs to be done once and for all.

\subsection{Bounding parameters}
The problem \eqref{minimisation_problem} is solved by an algorithm which performs a Nelder-Mead method. More precisely, we use \texttt{optim} procedure on the software R. Initial values for the parameters to 
be optimized over are required. To set those initial values, we propose to bound each parameter. For each of them, we give here a lower and an upper bound. 

As well as we observe the maxima, suppose we also have the minima : $I_{[t_{i-1},t_i[}$ (still for $(t_i)_{i\geq 0}$ a partition of $[0,T]$ of constant step $h \geq 0$). 
Then, the available quantities for bounding the parameters are the minima mean, denoted $m_{\text{min}}$; the maxima mean, denoted $m_{\text{max}}$; the smallest observed temperature, denoted 
$\text{rec}_\text{min}$ and the largest one, denoted $\text{rec}_\text{max}$.

Let us recall the OU process is assumed to be stationary, then we have, for all $t \geq 0$,
\begin{equation*}
\E[X_t]=\mu_0~~~~~~~~~~~~~~ \Var[X_t]=\frac{1}{2l_0} 
\end{equation*}
The expectation gives us natural bounds for the parameter $\mu_0$ : 
\[m_{\text{min}}\leq \mu_0 \leq m_{\text{max}}.\]
Moreover, we have $l=\frac{1}{2\Var[X_t]}.$
As for all $i$, $\abs{X_{t_i}-\bar X}\leq \max\left(\abs{\text{rec}_{\text{min}}-m_{\text{max}}};\abs{\text{rec}_{\text{max}}-m_{\text{min}}}\right)$, one may say that it is natural to upper bound the variance by \[\Var[X_t] \leq \left(\max\left(\abs{\text{rec}_{\text{min}}-m_{\text{max}}};\abs{\text{rec}_{\text{max}}-m_{\text{min}}}\right)\right)^2 \]
It then gives us a lower-bound $l_{\text{min}}$ for $\labee$. 

For the upper-bound, we use Theorem 2.7 of \cite{LiShao}. For all $x>0$, we have 
\[p_x : = \P\left(S_{[0,1[} - \E\left[S_{[0,1[}\right] \geq x\right) \leq \exp\left(-l_0x^2\right).\]
Hence \[l_0\leq l_{\text{max}} =\inf_{x>0}\left[\frac{-\ln(p_x)}{x^2}\right].\]
It remains to find the domain of $\beta_0$. First of all, $\beta_0 >0$.

Since $\beta_0=\frac{\langle X\rangle_T}{T}$, a classical estimator of $\beta_0$ (see \cite{Lebreton}) is $\hat \beta_0 = \frac{1}{T}\sum_{i=0}^{n-1}\left(X_{t_{i+1}}-X_{t_{i}}\right)^2$. 

Then, 
\[\beta_0 \leq \beta_\text{max}=\frac{1}{T}\sum_{i=1}^{n-1}\max\left[\left(S_{[t_{i},t_{i+1}[}-I_{[t_{i-1} ,t_i[}\right)^2,\left(I_{[t_{i},t_{i+1}[}-S_{[t_{i-1},t_i[}\right)^2\right]\]
Finally, $\widehat{\theta}_n=\text{argmin}_{\theta \in C} Q_n(\theta)$ where $C=\left[0,\beta_\text{max}\right]\times \left[m_{\text{min}},m_{\text{max}}\right]\times \left[l_{\text{min}},l_{\text{max}}\right] \subset \R \times \R_+^* \times \R_+^*$

\subsection{Parameters estimation on simulated data}
\label{sec:simulateddata} 
We are going to test our method on simulated data.  To choose realistic parameters, we use some temperatures data. The mean temperature leads us to take $\mu_0=22$. Using the difference between the maximal (respectively minimal) temperature and the mean temperature, we take $\sqrt{\frac{1}{2l_0}}=5$. The hourly 
correlation allows us to set $l_0\beta_0=0.95$ (see e.g. \cite{gringorten1968estimating} and \cite{BrodySyrokaZervos2002}). Then, $\theta_0=(47.5,22,0.02)$. 

We made several tests to make a compromise between the algorithm complexity and the precision of the estimation (with RMSE) which lead us to take $N_q=4$ here. Then, we take $s_1,s_2,s_3,s_4$ the empirical quantiles on the sample for $0.2,0.4,0.6$ and $0.8$ respectively so that there are uniformly distributed on the interval $[0,1]$. The $s_j$ values are settled and fixed for the whole estimation procedure.

50 samples are simulated over $T=n=1000$ days for each, with an Euler Scheme (time scale $\drm t =10^{-3}$ days), with $X_0 \sim \Ncal\left(22,25\right)$. The algorithm is launched on R twice on those 50 samples, once with the truncated samples over 100 days and one with the whole samples. The parameters found minimizing $Q_{100}$ and $Q_{1000}$ are presented in the following boxplots: 

\begin{figure}[H]
\centering
\includegraphics[width=12cm]{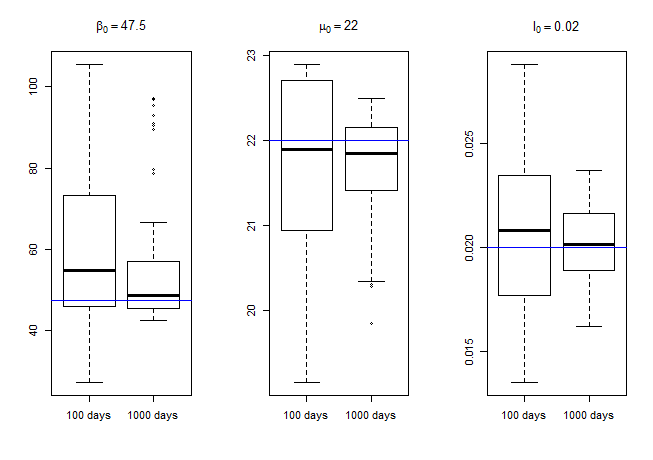}
\caption{Boxplots of the estimated parameters, the real value is indicated by the blue line.}
\label{boxplot}
\end{figure}
As we expected, the estimations are better on a larger sample. The relative RMSE for $\beta_0$, $\philaa$ and $\labee$ are respectively equal 0.4955, 0.04759 and 0.2194 for the small samples (100 days) and 0.4205, 0.03453 and 0.08928 for the larger samples (1000 days). The median parameters are satisfying. However, the parameters $\beta_0$ and $\mu_0$ seem biased. It appears that $\beta_0$ tends to be overestimated and on the contrary $\mu_0$ underestimated. Moreover, we observe a big variation in the estimators of $\beta_0$. It is confirmed by the relative RMSE (see above).
Better results are obtained if $\beta_0$ is fixed and performing a 2D-estimation, as in \cite{mullowney2008parameter}. Indeed, the relatives RMSE for $\philaa$ and $\labee$ are then respectively equal to 0.0107 and 0.0929. It is consistent with the results in \cite{martinbibby1995} where $\beta_0$ is assumed to be known.

\subsection{Real data}
\subsubsection{Parameters estimation}

In \cite{KleinTank2002}, daily temperature dataset in Paris through the ECA\&D project is provided (Data and metadata available at \url{http://www.ecad.eu}). This dataset is one of the longest in temperature measurement since it begins in 1900 but it records only maximum, minimum and mean daily temperature. In our application, we study daily summer temperature. In that way, we select maximal and minimal temperatures from 15th of june to the 14th of august ($61$ days) each year between 1950 and 1984 included, representing $35$ years of records and $2135$ days. These years are selected in order to avoid climate change influence so we can consider the dataset as maximum observations of a stationary process (see \cite{GIEC}). This is our train sample. We will keep the years after 1985 for the test sample. 

When we apply the estimation procedure presented in Section \ref{sec:estimprob}, we find $\widehat{\theta_0}=(34.35, 19.04, 0.02633)$. In order to assess the quality of this estimation, we propose to compare some theoretical quantities with empirical ones. That is done in the next section.

\subsubsection{Estimation validation}
To verify the estimation, we propose two models validation indicators: : comparison of quantiles and prediction.  The first validation indicator is just to check the quantile-quantile matching over our train sample. The second one is a validation using prediction and then does not use the train sample. 

\begin{figure}[H]
\centering
\includegraphics[width=10cm]{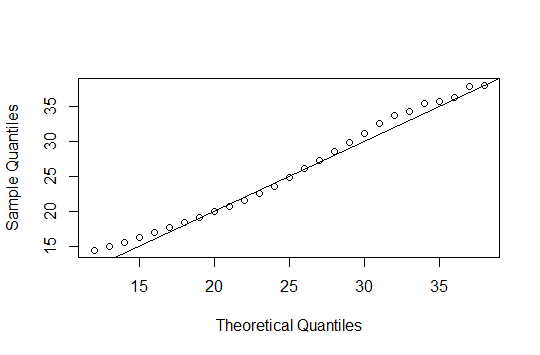}
\caption{Quantile-Quantile plot}
\label{qqplot}
\end{figure}
The first thing to check is the match of quantiles. To this aim, we draw a quantile-quantile plot (see Figure \ref{qqplot}). The plotted points fall near the line $y=x$  which indicates that the quantiles of the theoretical and data distributions agree.

We also want to assess the estimation quality by a prediction method. The estimation ends on the 14/08/1984. We take the mean temperature of the 14/06/1985 as an initial point to simulate processes (time scale $\drm t = 10^{-3}$ days) on 10 days with the estimated parameters. Then, we make a confidence interval by Monte Carlo simulations (1000 simulations) for the maxima over each of those days and compare it with the real values (between 15/06/1985 and 24/06/1985, in the test sample). 
\begin{figure}[H]
\centering
\includegraphics[width=10cm]{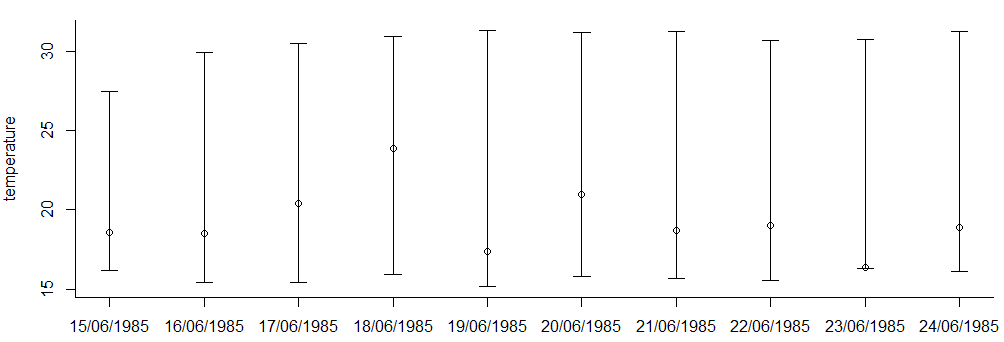}
\caption{Confidence limits at 95\% for the maxima between 15/06/1985 and 24/06/1985}
\end{figure}
We observe that the real values (the dots) are all in the confidence interval which confirms the pertinence of our model. We also made this test with 30 days and the results were also very good.

\subsection{Risk measures}

\subsubsection{Definitions}
Let us recall the goal of our study. We want to estimate some risk measures related to heat waves. Let us note that one may estimate any risk measure of his choice. Indeed, once the minimisation is performed, we can use a Monte-Carlo method simulating independent processes with estimated parameters.

There is two classical definition for a heat wave. The first one (see \cite{meteofrance}) is a sequence of consecutive days ($\Delta$ days) for which the maximum daily temperature is larger than a high-level threshold  ($a_\text{max}$) and the minimum daily temperature is greater than a low level one ($a_\text{min}$). Those temperatures thresholds ($a_\text{max}$, $a_\text{min}$) depend on the geographical zone. The second definition (see \cite{gringorten1968estimating}) is a sequence of consecutive days ($\Delta$ days) for which the minimum daily temperature is greater than a level ($a$).

As we have daily observations, to simplify the expression, we take $t_i=i \in \llbracket 0 , n \rrbracket $ and $S_{[0,1[}$ is then the supremum on the first day for example. \\
We define the two random variables 
\[m^S_{[i,i+\delta[}=\min\left(S_{[i,{i+1}[},\dots,S_{[i+\delta-1,i+\delta[}\right)\]
 and 
 \[m^I_{[i,i+\delta[}=\min\left(I_{[i,{i+1}[},\dots,I_{[i+\delta-1,i+\delta[}\right)\]
  for $\delta \in \llbracket 1, n-i \rrbracket $.\\ 
Then, we can express the probability of heat wave (for the first definition): \[\P \left(\exists i \in \llbracket 0, n-\Delta \rrbracket ,~m^S_{[i,i+\Delta[}\geq a_\text{max} \text{ , } m^I_{[i,i+\Delta[}\geq a_\text{min} \right).\]

Another interesting measure is the duration of an heat wave. Let us note, when there exists 
\[\tau_{\text{in}}= \min \{i \in \llbracket 0, n-\Delta \rrbracket,~m^S_{[i,i+\Delta[}\geq a_\text{max} \text{ , } m^I_{[i,i+\Delta[}\geq a_\text{min}\}\]
 and
 \[\tau_{\text{out}}= \tau_{\text{in}}+\max\{\delta \geq \Delta,~ m^S_{[i,i+\delta[}\geq a_\text{max} \text{ , } m^I_{[i,i+\delta[}\geq a_\text{min}\}.\] 
 Then, the mean duration of an heat wave is
\[\E\left[\tau_{\text{out}}-\tau_{\text{in}}~|~~m^S_{[\tau_{\text{in}},\tau_{\text{out}}[}\geq a_\text{max} \text{ , } m^I_{[\tau_{\text{in}},\tau_{\text{out}}[}\geq a_\text{min}\right]\] 

Using the second definition, we can also measure the severity of a heat wave with the area over the threshold for the first $\Delta$ days. Let $i_0$ be the first moment of a heat wave. This area is :
\[E=\E\left[\int_{i_0}^{i_0+\Delta} (X_s-a) \drm s~|~~m^I_{[i_0,i_0+\Delta[}\geq a \right]\]
As the process is assumed to be stationary, this quantity does not depend on $i_0$.

\subsubsection{Simulated data}
Since our goal is to estimate risk measures, we would like to see how they are impacted by the estimation of the parameters. We propose here to look at the last one, $E$, which uses the process itself.
\begin{figure}[H]
\centering
\includegraphics[width=4cm]{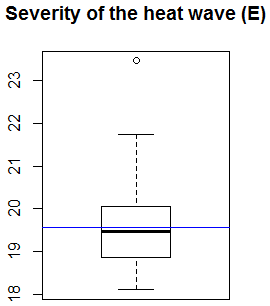}
\caption{Boxplot of the risk measure $E$ for the different estimated parameters. The level of the blue line ($19.57$) indicates the value for the real parameter $\theta_0$}
\label{box_espe}
\end{figure}
Figure \ref{box_espe} shows the boxplot of the risk measure obtained by Monte-Carlo simulation ($7\times 10^6$ simulations) for the estimated parameters found in Section \ref{sec:simulateddata} with $\Delta=3$ and $a=26.67$ (see \cite{gringorten1968estimating}). The relative RMSE for $E$ is 0.05291 which is satisfying.

\subsubsection{Real data}
We use the markers of Météo France for Paris (see  \cite{meteofrance}), we take $\Delta=3$, $a_\text{min}=21^\circ$C and $a_\text{max}=31^\circ$C. As we want to estimate the measures for a summer, we take $n=61$ days. Those measures are calculated with the estimated parameters on the real data, namely $\widehat{\theta_0}=(34.35, 19.04, 0.02633)$.

For the probability of heat wave, the Monte-Carlo method is performed with the simulation of $10^8$ years of $61$ days and we obtain a probability of $2.57\times 10^{-2}$ for a summer. There were 2 heat waves between 1985 to 2011 then a proportion of $7.41\times 10^{-2}$. This highlights the deviation of the temperatures in the last decades, due to climate change (\cite{GIEC}).

With $10^6$ simulations for the Monte-Carlo, we obtain a mean duration for an heat wave of 3.2 days. The 2 heat waves had lasted respectively 3 and 10 days.

\section{Conclusion and and future research directions}
In this paper, a new method to estimate the parameters of an OU process is proposed. Indeed, the proposed method includes a least square estimation based on the suprema observations. To this aim, the cdf of the suprema of an OU is given and theoretical results, including consistency of
model parameter estimation, are established. 

The numerical applications on real and simulated data prove the goodness of the estimation and its relevance. Risk measures such as the probability of heat wave or the duration of one have been studied and compared with the reality. The proposed model is also able to predict temperatures for a few days. 

Some directions for further investigations are summarized as follows. For example, in continuity with this work, obtaining explicit expressions of risk measures may be interesting in the model. To this aim, one may know the joint law of the supremum and the process. Moreover, another interesting estimation for the parameters of the process might be done using Maximum Simulated Likelihood Estimation (see \cite{MSLE}).

\section{Appendix}
\label{annexe}

\begin{figure}[H]
\centering
\includegraphics[scale=0.55]{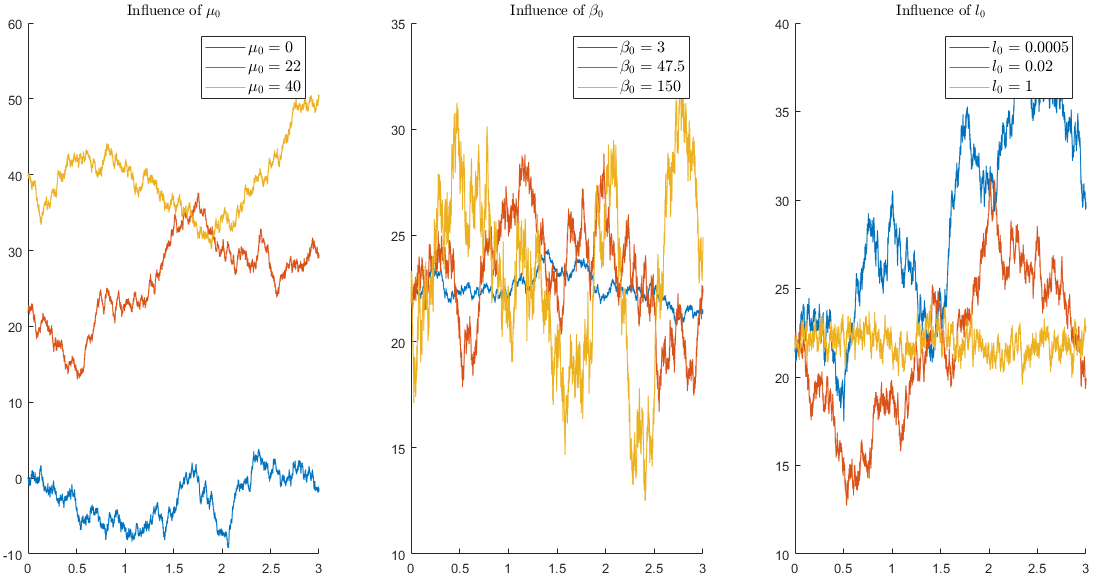}
\caption{Different trajectories of Ornstein-Uhlenbeck processes for different parameters. When not precised, the other parameters are $\mu_0=22$, $l_0=0.02$ and $\beta_0=47.5$. The time scale is $\drm t = 10^{-4}$ days. For each trajectories, $X_0=\mu_0$. } 
\label{influparam}
\end{figure}
\bibliographystyle{alpha}
\bibliography{bibarticle}

             \end{document}